\documentclass[12pt]{article}

\usepackage{amsmath}
\usepackage{amsfonts}
\usepackage{graphicx}
\usepackage{amsthm}
\usepackage{amssymb}

\newtheorem{theorem}{Theorem}
\newtheorem{lemma}{Lemma}[theorem]
\newtheorem{corollary}{Corollary}

\newcommand{\ds}{\displaystyle}
\numberwithin{equation}{section}

\title{Asymptotics of the coefficients of polynomials,  asymptotic zero distribution and free probability}
\author{Walter Van Assche \\ KU Leuven, Belgium}
\date{\today}  

\begin{document}

  \maketitle

\begin{abstract}
The asymptotic behavior of the coefficients of a sequence of polynomials is given under the conditions that all the zeros are real and positive and the zeros
have an asymptotic zero distribution on $(0,\infty)$. We show how the limit for the coefficients is related to the Stieltjes transform of the asymptotic zero distribution and make a connection with the  $S$-transform in free probability. We illustrate with some examples involving orthogonal and multiple orthogonal polynomials.
\end{abstract}

\section{Introduction}
Recently Mart\'\i nez-Finkelshtein, Morales and Perales \cite{M-FMP} showed how one can find the asymptotic zero distribution of some multiple orthogonal polynomials
by observing that these multiple orthogonal polynomials are finite free convolutions of polynomials for which the asymptotic zero distribution is known. 
In this paper a different approach is used to find the asymptotic zero distribution by using a result about the asymptotic behavior of the coefficients of polynomials
and by connecting this to the $S$-transform, an important transform in free probability theory (see, e.g., \cite{Speicher}). The $S$-transform often turns out to be the product of simpler $S$-transforms,
so that the asymptotic zero distribution is a free multiplicative convolution of known measures, often appearing as the asymptotic zero distribution of classical orthogonal polynomials.

\subsection{The $S$-transform}
 In this paper we mainly use measures on $[0,\infty)$ because our main theorem (Theorem~\ref{thm:main}) requires the zeros to be positive.
Suppose $\mu$ is a probability measure on $[0,\infty)$.
The \textit{Stieltjes transform} of $\mu$ (or Cauchy transform, often also called Markov function) is given by
\[    G(z) = \int_0^\infty \frac{d\mu(x)}{z-x}, \qquad z \in \mathbb{C} \setminus [0,\infty). \]
In this paper we mainly use measures on $[0,\infty)$ because our main theorem (Theorem~\ref{thm:main}) requires the zeros to be positive.
The \textit{moment generating function} (without zero-th moment) is given by
\[    M(z) = \frac{1}{z} G(1/z) -1 = \int_0^\infty \frac{zx}{1-zx} \, d\mu(x) = \sum_{n=1}^\infty m_n z^n,  \]
where the moments are given by
\[     m_n = \int_0^\infty  x^n \, d\mu(x).  \]
We assume that all the moments are finite and that $M$ has radius of convergence $R > 0$. Note that $m_0=1$ and $M(0)=0$.
The \textit{$S$-transform} is then given by
\[     S(u) = \frac{u+1}{u} M^{-1}(u).  \]
This function is well-defined in a neighborhood of $0$.
If the $S$-transform of a measure $\mu$ is known in a neighborhood of $0$, then one can find the moment generating function $M$ in a neighborhood
of $0$ and hence also the Stieltjes transform $G$ in a neighborhood of $\infty$. By using the analytic extension of $G$ and the Stieltjes inversion theorem, one can then determine the measure $\mu$ on the real line.

An important property of the $S$-transform is the following: if $S_1$ and $S_2$ are $S$-transforms of two measures $\mu_1$ and $\mu_2$, then the product
$S(z)=S_1(z)S_2(z)$ is also the $S$-transform of a measure, which is known as the free multiplicative convolution $\mu = \mu_1 \boxtimes \mu_2$.

\section{Main result}

Let $\{ p_n(x), n \in \mathbb{N}\}$ be monic polynomials with only positive zeros
\[   p_n(x) = (x-x_{1,n})(x-x_{2,n}) \cdots (x-x_{n,n}) = \sum_{k=0}^n (-1)^k a_{k,n}  x^{n-k}. \]
The relation between the zeros $\{x_{i,n}, 1 \leq i \leq n\}$ and the coefficients of the polynomial $p_n$ is well known (Vi\`ete relations):
\begin{eqnarray*}   
                       a_{0,n} &=& 1, \\
                       a_{1,n} &=&  \sum_{j=1}^n x_{j,n}, \\[-10pt]
                       &\vdots & \\[-2pt]
                       a_{k,n} & = & \sum_{i_1 < i_2 < \cdots < i_k} x_{i_1,n}x_{i_2,n}\cdots x_{i_k,n}   
                                    =\frac{1}{k!} \sum_{i_1\neq i_2 \neq \cdots \neq i_k}  x_{i_1,n}x_{i_2,n}\cdots x_{i_k,n},   \\[-10pt]  
                                    & \vdots & \\[-2pt]
                       a_{n,n} &=& x_{1,n}x_{2,n} \cdots x_{n,n}.
\end{eqnarray*}
This relation between the coefficients and the zeros  is highly non-linear. In \cite{WVAFO} the following result was proved for polynomials with only negative real zeros.
The formulation given below is for polynomials with only positive real zeros. \newpage

\begin{theorem}[Van Assche-Fano-Ortolani 1987]  \label{thm:main}
Let $\mu_n$ be the zero counting measure
\[   \mu_n = \frac{1}{n}  \sum_{k=1}^n \delta_{x_{k,n}}. \]
Suppose that $\mu_n \to \mu$ in weak convergence and $\mu$ has no mass at $0$. Then there exists a concave and
differentiable function $g$ on $[0,1]$ such that
\begin{equation}  \label{lim-g}
      \lim_{\frac{k}{n} \to t} \frac1n \log a_{k,n} = g(t), 
\end{equation}
and
\begin{equation}    \label{lim-f}
    \lim_{\frac{k}{n} \to t} \frac{a_{k,n}}{a_{k-1,n}} = e^{g'(t)} = f(t).  
\end{equation}
The inverse of $f$ is given by
\[   t = f^{-1}(x) = \int_0^\infty \frac{y}{x+y}\, d\mu(y) = 1 - x \int_0^\infty \frac{d\mu(y)}{x+y},  \]
and
\[   g(t) = -(1-t) \log f(t) + \int_0^\infty \log[f(t)+y]\, d\mu(y).  \]
\end{theorem}

An important consequence is
\begin{corollary}  \label{cor:S}
Let $\mu$ be the asymptotic distribution of the zeros of $p_n$, then
%\[     \widehat{M}^{-1}(u) = \frac{-1}{f(-u)}, \]
 the $S$-transform of $\mu$ is
\[    S(u) = -\frac{u+1}{u} \frac{1}{f(-u)}.   \]
\end{corollary}
This will be our main tool for obtaining the asymptotic distribution of the zeros of some orthogonal and multiple orthogonal polynomials.

We will give a sketch of the proof of Theorem \ref{thm:main} with the main ingredients. For the full proof we refer to \cite{WVAFO}.
\begin{proof}
Introduce a sequence of continuous functions\footnote{Observe that in \cite{WVAFO} there is a misprint in equation (13), which has $\frac{j}{n+1}$ instead of $\frac{j+1}{n}$.}
\[  g_n(t) = \begin{cases}     \ds \frac{1}{n} \log a_{j,n}, & \ds \textrm{if } t = \frac{j}{n}, \\[10pt]
               \ds   \frac{1}{n} \log a_{j,n} + \left( t- \frac{j}{n} \right) ( \log a_{j+1,n} -\log a_{j,n}), & \ds \textrm{if } \frac{j}{n} < t < \frac{j+1}{n}.
  \end{cases} \]
Then one can show that  the sequence $\bigl(g_n)_{n \in \mathbb{N}}$ converges uniformly on every compact subset of $(0,1)$ to a concave continuous function $g$.
This gives the existence of a function $g$ for which \eqref{lim-g} holds. 
Observe that
\[   \prod_{k=1}^j \frac{a_{k,n}}{a_{k-1,n}} = a_{j,n} \qquad (a_{0,n} = 1). \]
Taking logarithms gives
\[    \sum_{k=1}^j  \log \frac{a_{k,n}}{a_{k-1,n}} = \log a_{j,n}. \]
This can be rewritten as
\[     \int_0^t \log \frac{a_{[sn],n}}{a_{[sn]-1,n}} \, ds = \frac{1}{n}  \log a_{j,n},   \qquad    j = [nt] .  \]
Taking the limit $n \to \infty$ then gives
\[   \int_0^t  \log f(s)\, ds = g(t).  \]
Differentiate this to find
\[      \log f(t) = g'(t) , \]
so that $f(t) = e^{g'(t)}$.  

Note that 
\[           (-1)^n p_n(-x) = \sum_{k=0}^n a_{k,n} x^{n-k} > 0 , \qquad x > 0.  \]
so that for every $k$
\[                            a_{k,n} x^{n-k} \leq  (-1)^n p_n(-x)     \]
and thus also
\[               \max_{0 \leq k \leq n}     a_{k,n} x^{n-k} \leq  (-1)^n p_n(-x).     \]
On the other hand
\[                                      (-1)^n p_n(-x) \leq (n+1)   \max_{0 \leq k \leq n}     a_{k,n} x^{n-k} .  \]
Taking logarithms gives
\begin{multline*}       
   \max_{0 \leq k \leq n}  \left(  \log a_{k,n} + (n-k) \log x \right)  \leq  \log (-1)^n p_n(-x)  \\
 \leq \log (n+1) +     \max_{0 \leq k \leq n}     \left(\log a_{k,n} + (n-k) \log x \right).  
 \end{multline*}
Divide by $n$ and let $n \to \infty$ to find
\[  \max_{0 \leq t \leq 1} \bigl( g(t) + (1-t) \log x \bigr) = \int_0^\infty \log (x+y) \, d\mu(y).    \]
The maximum is attained at  $t^*$ for which $g'(t^*) - \log x = 0$ (there is only one maximum on $[0,1]$).
Let $f(t) = e^{g'(t)}$ then $f$ is a decreasing function. 
Call its inverse $h$ so that $f(t)=x$ corresponds to $h(x) = t$.
Note that $g'(h(x)) = \log x$ so that the maximum is attained at $t^* = h(x)$ and
\begin{equation}   \label{gh}
        g(h(x)) + (1-h(x)) \log x = \int_0^\infty \log(x+y)\, d\mu(y).  
\end{equation}
This gives
\[     g(t) + (1-t) \log f(t) = \int_0^\infty \log[f(t)+y]\, d\mu(y).   \]
Taking the derivative of \eqref{gh}  gives
\[   g'(h(x)) h'(x) - h'(x) \log x + \frac{1-h(x)}{x} = \int_0^\infty \frac{d\mu(y)}{x+y} , \]
which is the same as
\[       1-h(x) = \int_0^\infty \frac{x}{x+y} \, d\mu(y)    \]
or
\[    f^{-1}(x) = 1-  \int_0^\infty \frac{x}{x+y} \, d\mu(y).  \]
\end{proof}

\section{Examples}
\subsection{One multiple zero}   \label{ex:3.1}
A very simple example is the sequence of polynomials $p_n(x)=(x-a)^n$, with $a>0$. Then
\[   p_n(x) = \sum_{k=0}^n \binom{n}{k} (-1)^k a^k x^{n-k},   \]
so that
\[   a_{k,n} = a^k \binom{n}{k}  \] 
and
\[  \lim_{\frac{k}{n} \to t}   \frac{a_{k,n}}{a_{k-1,n}} = \lim_{\frac{k}{n} \to t} a \frac{n-k+1}{k} = a \frac{1-t}{t}. \]
The function $f$ is thus given by
\[ f(t) = a \frac{1-t}{t}, \]
and by Corollary \ref{cor:S}
\[    S(u) =  - \frac{u+1}{u} \frac{-u}{1+u} \frac{1}{a} = \frac{1}{a}.  \]
This is the $S$-transform of a Dirac measure at $a$.
Furthermore we have
\[    \lim_{\frac{k}{n} \to t} \frac1n \log a_{k,n} = t \log a - t \log t - (1-t) \log(1-t), \]
so that the function $g$ is given by
\[   g(t) = t \log a -t \log t -(1-t) \log(1-t). \]
This is indeed a concave and differentiable function on $[0,1]$, see Figure \ref{fig1}.

\begin{figure}[ht]
\centering
\includegraphics[width=5in]{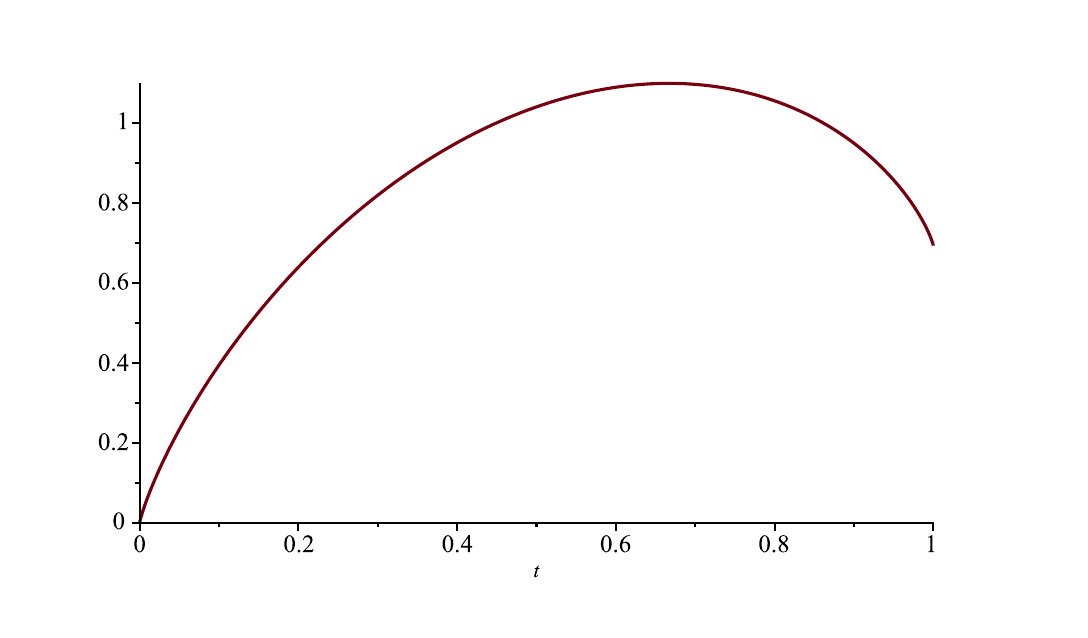}
\caption{The function $g$ for Example \ref{ex:3.1}  (with $a=2$).}
\label{fig1}
\end{figure}

\subsection{Laguerre polynomials}
The Laguerre polynomials $L_n^{(\alpha)}$ are classical orthogonal polynomials on $[0,\infty)$ for the gamma distribution with density $w(x)=x^\alpha e^{-x}$ ($\alpha > -1$).
A natural scaling is to look at the zeros of $L_n^{(\alpha)}(nx)$, so that we look at the asymptotic distribution of the scaled zeros $\{ x_{k,n}/n, 1 \leq k \leq n \}$.
An explicit expression for the monic Laguerre polynomials is \cite[Eq. (5.1.6)-(5.1.8)]{Szego}
\[   L_n^{(\alpha)}(nx) = \frac{n!}{n^n} \sum_{k=0}^n \binom{n+\alpha}{k} \frac{(-1)^k}{(n-k)!} (nx)^{n-k}.  \]
The coefficients are thus given by
 \[   a_{k,n} = \frac{n!}{n^k} \binom{n+\alpha}{k} \frac{1}{(n-k)!},   \]
and
 \[  \lim_{\frac{k}{n} \to t}   \frac{a_{k,n}}{a_{k-1,n}} =  \lim_{\frac{k}{n} \to t} \frac{(n+\alpha-k+1)(n-k+1)}{kn} = \frac{(1-t)^2}{t}.  \]
The function $f$ is therefore
\[ f(t) = \frac{(1-t)^2}{t}   \]
and by Corollary \ref{cor:S} the $S$-transform of the asymptotic distribution of the scaled zeros is
 \[   S(u) = - \frac{u+1}{u} \frac{-u}{(1+u)^2} = \frac{1}{u+1}.  \]
This is the $S$-transform of the \textit{Marchenko-Pastur distribution} on $[0,4]$ (Figure \ref{fig2}) with density
 \[    w(x) = \frac{1}{2\pi} \frac{\sqrt{4-x}}{\sqrt{x}}, \qquad x \in [0,4].  \]

\begin{figure}[ht]
 \centering
 \includegraphics[width=5in]{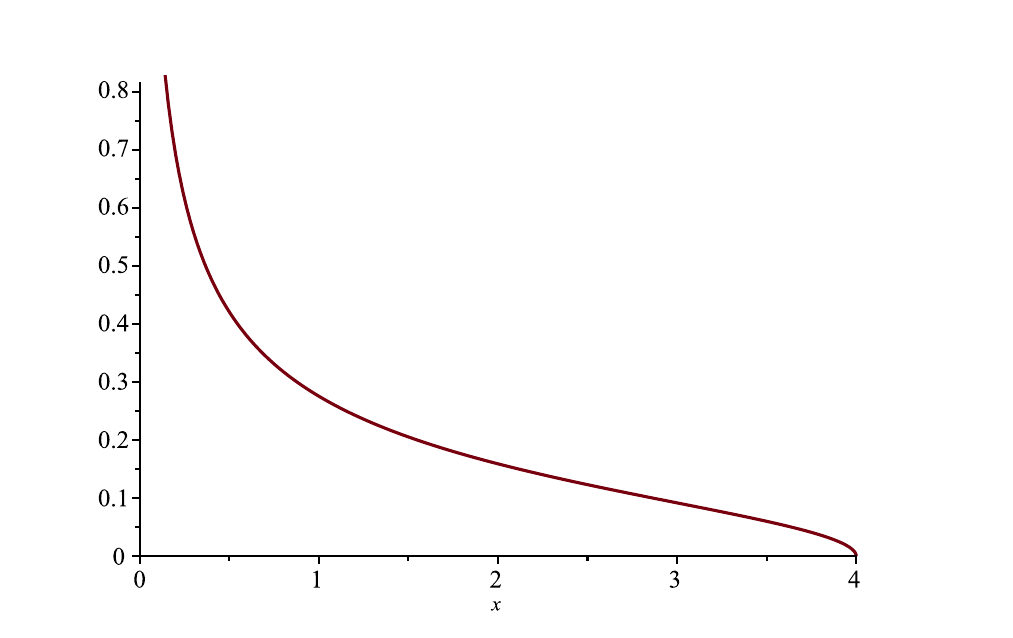}
 \caption{The Marchenko-Pastur distribution on $[0,4]$.}
 \label{fig2}
\end{figure}

We can also find what happens when the parameter $\alpha$ depends on the degree $n$. 
For Laguerre polynomials $L_n^{(\alpha)}(nx)$ with $\alpha/n \to a > 0$ one has
 \begin{eqnarray*}   \lim_{\frac{k}{n} \to t}   \frac{a_{k,n}}{a_{k-1,n}} &=&  \lim_{\frac{k}{n} \to t} \frac{(n+\alpha-k+1)(n-k+1)}{kn}  \\ 
  &=& \frac{(1+a-t)(1-t)}{t},  
  \end{eqnarray*} 
so that
\[ f(t) = \frac{(1+a-t)(1-t)}{t}.  \]
The $S$-transform then becomes
\[   S(u) = -\frac{u+1}{u} \frac{-u}{(1+a+u)(1+u)} = \frac{1}{1+a+u}.  \]
This is the \textit{Marchenko-Pastur distribution} MP$(a)$ on $[(\sqrt{a+1}-1)^2,(\sqrt{a+1}+1)^2]=[ \lambda_-,\lambda_+]$, where $a >0$, see Figure \ref{fig3}.
 The density is
\[   w(x) =  \frac{1}{2\pi} \frac{\sqrt{(x-\lambda_-)(\lambda_+-x)}}{x}, \qquad   \lambda_- < x < \lambda_+ ,  \]
with
\[  \lambda_{\pm} = (\sqrt{a+1} \pm 1)^2.  \]

\begin{figure}[ht]
 \centering
 \includegraphics[width=5in]{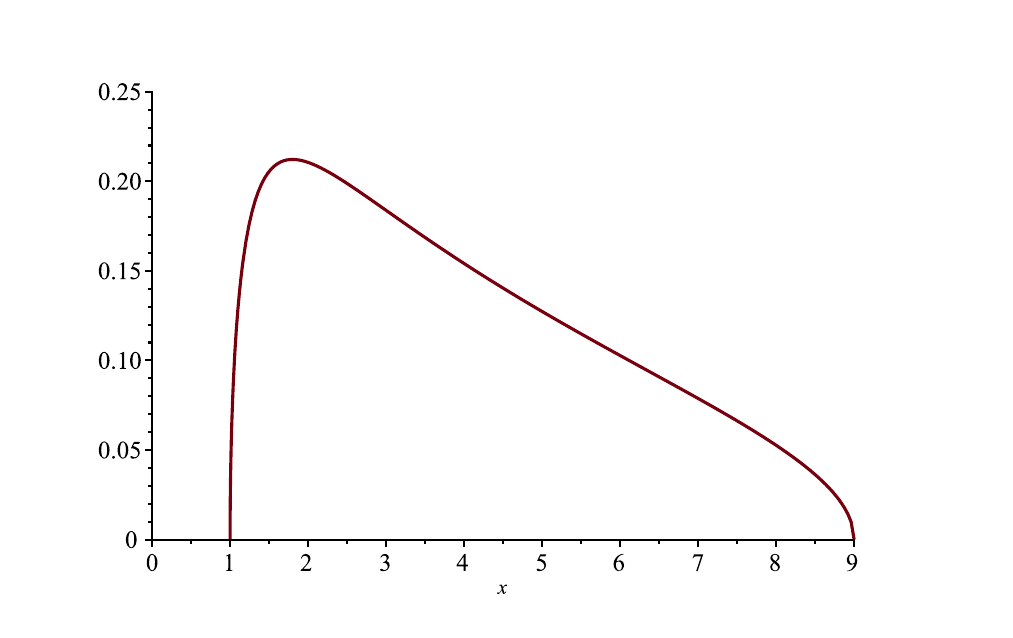}
\caption{The Marchenko-Pastur distribution with $a=3$.}
\label{fig3}
\end{figure}

Observe that the zeros have moved away from the origin $0$.
The asymptotic distribution of the zeros of Laguerre polynomials $L_n^{(an+\alpha)}(nx)$ was given by Gawronski in \cite{Gawronski}. 

\subsection{Jacobi-Pi\~neiro polynomials}

\begin{figure}[ht]
\centering
\includegraphics[width=5in]{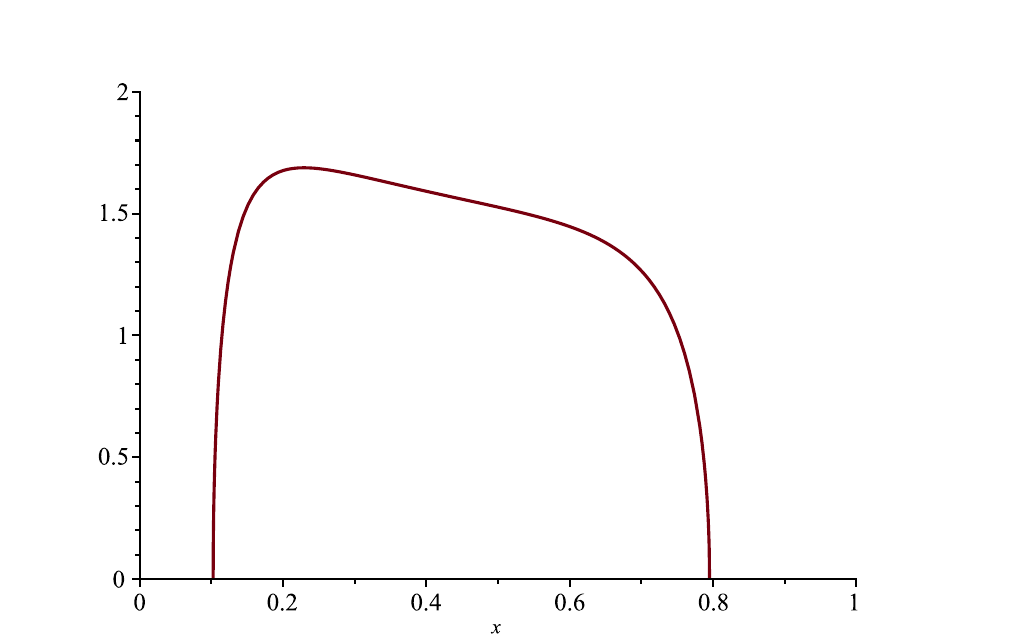}
\caption{The Kesten-McKay distribution for $a=2$ and $b=3$. }
\label{fig4}
\end{figure}
Now an example with multiple orthogonal polynomials. For the type II Jacobi-Pi\~neiro polynomials one has the orthogonality relations 
\[     \int_0^1 P_{\vec{n}}(x) x^{\alpha_j} (1-x)^\beta x^k\, dx = 0, \qquad 0 \leq k \leq n_j-1, \quad 1 \leq j \leq r, \]
where $r \geq 1$. We use the multi-index  $\vec{n} = (n_1,\ldots,n_r)$ and the polynomial has degree $|\vec{n}|=n_1+\cdots+n_r$.
For $\beta=0$ one has the explicit expression \cite[Eq. (23.3.5) on p. 627]{Ismail}
\[   P_{\vec{n}}(x) = C_{\vec{n}}  \  {}_{r+1}F_r\left( \left. \begin{array}{c} -|\vec{n}|,\alpha_1+n_1+1,\ldots,\alpha_r+n_r+1 \\
                                                                                                                    \alpha_1+1,\ldots,\alpha_r+1  \end{array} \right| x \right), \]
with $C_{\vec{n}}$ a normalizing constant that makes this a monic polynomial. The coefficients are thus given by
\[   a_{k,\vec{n}} = C_{\vec{n}} \binom{|\vec{n}|}{k} \frac{(\alpha_1+n_1+1)_{|\vec{n}|-k} \cdots (\alpha_r+n_r+1)_{|\vec{n}|-k}}
                                                         {(\alpha_1+1)_{|\vec{n}|-k} \cdots (\alpha_r+1)_{|\vec{n}|-k}} ,  \]
and one has
\[    \frac{a_{k,\vec{n}}}{a_{k-1,\vec{n}}} = \frac{|\vec{n}|-k+1}{k} \frac{(\alpha_1+1+|\vec{n}|-k)\cdots(\alpha_r+1+|\vec{n}|-k)}
                                           {(\alpha_1+n_1+1+|\vec{n}|-k)\cdots(\alpha_r+n_r+1+|\vec{n}|-k)} . \]                            
If $n_j/|\vec{n}| \to q_j$ as $|\vec{n}| \to \infty$, where $q_1+\cdots +q_r=1$, then
\[   \lim_{k/|\vec{n}| \to t}  \frac{a_{k,\vec{n}}}{a_{k-1,\vec{n}}} = \frac{1-t}{t} \frac{(1-t)^r}{(1+q_1-t)\cdots(1+q_r-t)}.  \]
The function $f$ is given by
\[    f(t) = \frac{1-t}{t} \frac{(1-t)^r}{(1+q_1-t)\cdots(1+q_r-t)}, \]
and by Corollary \ref{cor:S} the $S$-transform of the asymptotic distribution of the zeros is
\[   S(u) =  \prod_{j=1}^r \frac{1+q_j+u}{1+u}.  \]
This is the product of $r$ easier $S$-transforms
\[  S_j(u) = \frac{1+q_j+u}{1+u} , \]
and each of these is the $S$-transform of the  distribution
\[   \mu_j = (1-q_j) \delta_1 + q_j \textup{KM}\left( 0,\frac{1-q_j}{q_j}\right), \]
where KM$(a,b)$ is the \textit{Kesten-McKay distribution} with parameters $a,b > 0$. 
Hence the asymptotic distribution of the zeros is the $r$-fold free multiplicative convolution
\[    \mu = \mu_1 \boxtimes \mu_2 \boxtimes \cdots \boxtimes \mu_r.  \]
The density of KM$(a,b)$, $a,b \geq 0$ is given by (see Figure \ref{fig4})
\[    w(x)  = \frac{2}{C\pi} \frac{\sqrt{(x-\lambda_-)(\lambda_+-x)}}{x(1-x)} , \qquad \lambda_- < x < \lambda_+, \]
with
\[   \lambda_{\pm} = \frac{(a+1)(a+b+1)+b+1 \pm 2 \sqrt{(a+1)(b+1)(a+b+1)}}{(a+b+2)^2}.  \]
The normalization constant is
\[  C = (\sqrt{\lambda_+}-\sqrt{\lambda_-})^2 + (\sqrt{1-\lambda_-}-\sqrt{1-\lambda_+})^2.  \]

The density for the asymptotic distribution for general $r$ and $q_1=\ldots=q_r=1/r$ was obtained in \cite[Theorem 1.1]{NeuschelWVA},
 see also \cite[Theorem D]{NikiSor} for a similar result. 
The asymptotic zero distribution $r=2$ and $q_1=q_2=1/2$ has density (Figure \ref{fig6})
\[ w(x) = \frac{\sqrt{3}}{4\pi} \, \frac{\left(1+\sqrt{1-x}\right)^{\frac{1}{3}}+\left(1-\sqrt{1-x}\right)^{\frac{1}{3}}}{  \,x^{\frac{2}{3}} \sqrt{1-x}},  \]
see, e.g., \cite[Thm. 2.5]{Cous2WVA}. Note that for $r=1$ (in which case $q_1=1$) the $S$-transform is
\[    S(u) = \frac{2+u}{1+u},  \]
which is the $S$-transform of KM$(0,0)$ and this is the arcsine distribution on $[0,1]$ with density
\[      w(x) = \frac{1}{\pi \sqrt{x(1-x)}}, \qquad x \in (0,1).   \]

\begin{figure}[ht]
\centering
\includegraphics[width=5in]{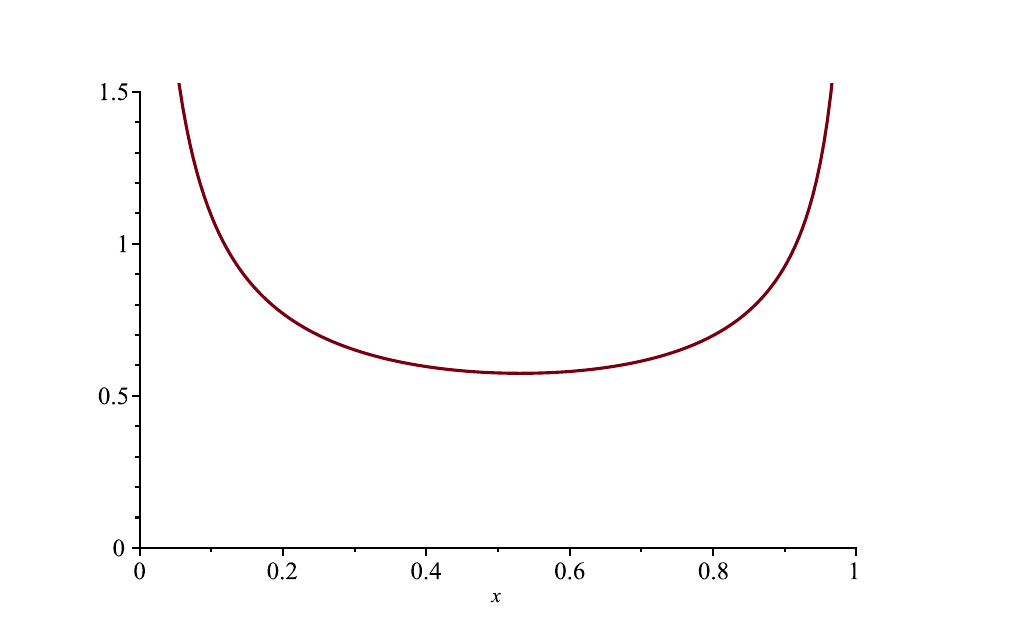}
\caption{The asymptotic zero density for Jacobi-Pi\~neiro polynomials ($r=2$) with $q_1=q_2=1/2$.}
\label{fig6}
\end{figure}

We can also find what happens if the parameters $\alpha_j$ depend on the degree of the polynomial.     
If $\alpha_j/|\vec{n}| \to a_j >0$ then
\[   \lim_{k/|\vec{n}| \to t}  \frac{a_{k,\vec{n}}}{a_{k-1,\vec{n}}} = \frac{1-t}{t} \frac{(a_1+1-t)\cdots(a_r+1-t)}{(a_1+1+q_1-t)\cdots(a_r+1+q_r-t)}, \]
so that the function $f$ is
\[   f(t) = \frac{1-t}{t} \frac{(a_1+1-t)\cdots(a_r+1-t)}{(a_1+1+q_1-t)\cdots(a_r+1+q_r-t)}.  \]
The $S$-transform then becomes
\[   S(u) = \prod_{j=1}^r \frac{a_j+q_j+1+u}{a_j+1+u}, \]
and this is the $S$-transform of the $r$-fold free multiplicative convolution
\[     \mu =  \mu_1 \boxtimes \mu_2 \boxtimes \cdots \boxtimes \mu_r  \]
with
\[     \mu_j = (1-q_j) \delta_1 + q_j \textup{KM}\left( \frac{a_j}{q_j}, \frac{1-q_j}{q_j} \right).   \]
Note that for $r=1$ the asymptotic distribution of the zeros of the Jacobi polynomials $P_n^{(an+\alpha,bn+\beta)}$ was obtained by Gawronski and Shawyer 
\cite{Gawronski-Shawyer}. Again the zeros have moved away from the origin.

For the type I Jacobi-Pi\~neiro polynomials with $r=2$ one has polynomials $A_{n,m}$ (of degree $n-1$) and $B_{n,m}$ (of degree $m-1$) for which
\[    \int_0^1 [A_{n,m}(x)x^{\alpha_1} + B_{n,m}(x)x^{\alpha_2}] (1-x)^\beta x^k\, dx = 0, \qquad 0 \leq k \leq n+m-2. \]
An explicit expression is given in \cite{BDFM}:
\[     A_{n,m}(x) = C_{n,m} \  {}_3F_2 \left( \left. \begin{array}{c}  -n+1, \alpha_1+\beta+n+m,\alpha_1-\alpha_2-m+1 \\
                                                                                              \alpha_1+1, \alpha_1-\alpha_2+1  \end{array} \right| x \right) , \]
\[     B_{n,m}(x) = D_{n,m} \  {}_3F_2 \left( \left. \begin{array}{c}  -m+1, \alpha_2+\beta+n+m,\alpha_2-\alpha_1-n+1 \\
                                                                                              \alpha_2+1, \alpha_2-\alpha_1+1  \end{array} \right| x \right).  \]
These polynomials may have complex zeros. If $n=m$ then the zeros are negative, so therefore we look at the zeros of $A_{n+1,n+1}(-x)$, the analysis for $B_{n+1,n+1}(-x)$
is similar. We have
\[    A_{n+1,n+1}(-x) = \sum_{k=0}^n (-1)^k a_{k,n} x^{n-k}, \]
with
\[     a_{k,n} = C_{n+1,n+1} (-1)^{n-k}\binom{n}{k} \frac{(\alpha_1+\beta+2n+2)_{n-k} (\alpha_1-\alpha_2-n)_{n-k}}{(\alpha_1+1)_{n-k}(\alpha_1-\alpha_2+1)_{n-k}}, \]
with $C_{n+1,n+1}$ a normalizing constant. We then find (for $\alpha_1$ and $\alpha_2$ fixed)
\[   \lim_{k/n\to t} \frac{a_{k,n}}{a_{k-1,n}} = \frac{1-t}{t} \frac{(1-t)^2}{(3-t)t}, \]
so that
\[    f(t) = \frac{1-t}{t} \frac{(1-t)^2}{(3-t)t}.   \]
The $S$-transform is then
\[   S(u) = \frac{-u(3+u)}{(1+u)^2} = \frac{3+u}{1+u} \ \frac{-u}{1+u}.  \]
The asymptotic distribution of the zeros of $A_{n+1,n+1}(-x)$ is the free multiplicative convolution $\mu_1 \boxtimes \mu_2$, 
where $\mu_1$ is the Kesten-McKay distribution KM$(0,1)$  and $\mu_2$ the measure on $[0,\infty)$ with density (see Figure \ref{fig5})
\begin{equation}  \label{eq:JPI}
   w(x) = \frac{1}{\pi} \frac{1}{\sqrt{x}(x+1)}, \qquad x > 0. 
\end{equation} 
Note that the measure $\mu_2$, and hence also $\mu = \mu_1 \boxtimes \mu_2$, has unbounded support $[0,\infty)$ which means that the zeros of $A_{n+1,n+1}$ are dense on $(-\infty,0]$.

\begin{figure}[ht]
\centering
\includegraphics[width=5in]{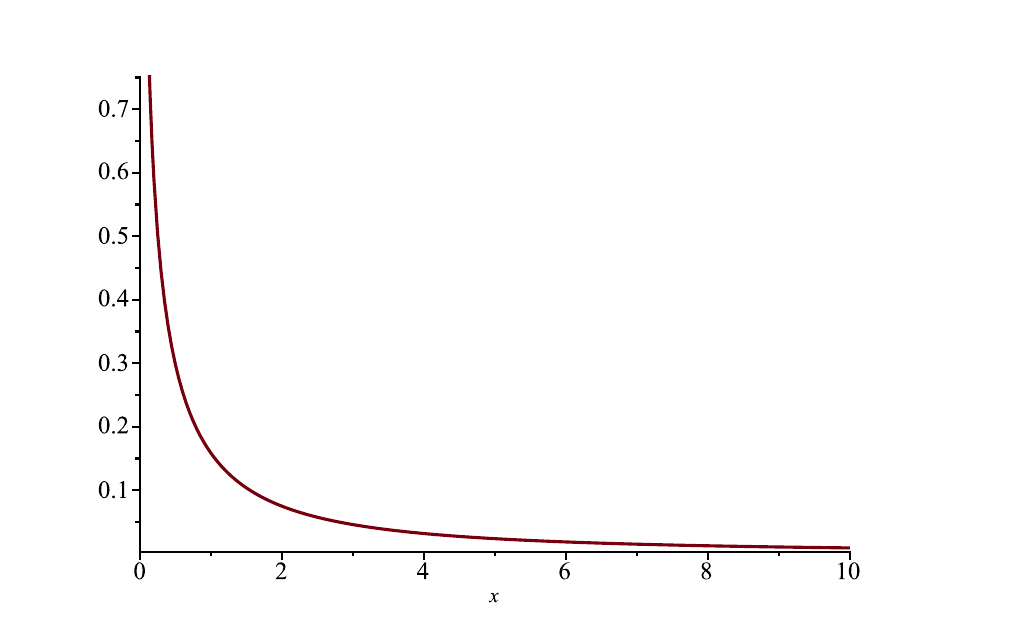}
\caption{The density \eqref{eq:JPI} for the measure $\mu_2$ for type I Jacobi-Pi\~neiro polynomials.}
\label{fig5}
\end{figure}

\subsection{Multiple OP with modified Bessel weights $(K_\nu,K_{\nu+1})$}  \label{ex:3.4}

These multiple orthogonal polynomials were introduced in \cite{WVA-Yakubovich}. We restrict the multiple orthogonal polynomials to the stepline,
i.e., we only use the multi-indices $(n,n)$ and $(n+1,n)$, with $p_{2n}(x)=P_{n,n}(x)$ and $p_{2n+1}(x) = P_{n+1,n}(x)$, and the multiple orthogonality is
\begin{eqnarray*}
   \int_0^\infty p_n(x) x^{\alpha+\nu/2} K_\nu(2\sqrt{x}) x^k \, dx = 0, && \qquad 0 \leq k \leq \lceil \frac{n}2 \rceil \leq -1, \\
   \int_0^\infty p_n(x) x^{\alpha+\nu/2} K_{\nu+1}(2\sqrt{x}) x^k \, dx = 0, && \qquad 0 \leq k \leq \lfloor \frac{n}2 \rfloor \leq  -1.
\end{eqnarray*}
There is an explicit hypergeometric expression for these polynomials (see \cite[Theorem 2]{ECousWVA})
\[    p_n(x) = C_n \ {}_1F_2 \left( \left. \begin{array}{c} -n \\ \alpha+1, \alpha+\nu+1 \end{array} \right| x \right) . \]
The appropriate scaling of the zeros is $n^2$, and the coefficients of $p_n(n^2x)$ are
\[      	a_{k,n} = C_{n} \binom{n}{k} \frac{n^{2n-2k}}{(\alpha+1)_{n-k}(\alpha+\nu+1)_{n-k}}, \]
so that
\[   \lim_{k/n \to t} \frac{a_{k,n}}{a_{k-1,n}} = \frac{1-t}{t} (1-t)^2. \]
The function $f$ is therefore given by
\[   f(t) =  \frac{1-t}{t} (1-t)^2, \]
and by Corollary \ref{cor:S} the $S$-transform is
\[   S(u) = \frac{1}{(1+u)^2}. \]
Hence the asymptotic distribution of the scaled zeros $\{ x_{k,n}/n^2, 1 \leq k \leq n\}$ is $\mu_1 \boxtimes \mu_1$, where $\mu_1$ is the {Marchenko-Pastur} distribution
 (on $[0,4]$).
The density of $\mu_1 \boxtimes \mu_1$ is $\frac{4}{27} w(4x/27)$ with
\[  w(x) = \frac{3 \sqrt{3}}{4\pi} \, \frac{\left(1+\sqrt{1-x}\right)^{\frac{1}{3}}-\left(1-\sqrt{1-x}\right)^{\frac{1}{3}}}{  x^{\frac{2}{3}}} , \]
 see Figure \ref{fig7}.
This was obtained in \cite[Thm. 2.7]{Cous2WVA}
\begin{figure}[ht]
\centering
\includegraphics[width=5in]{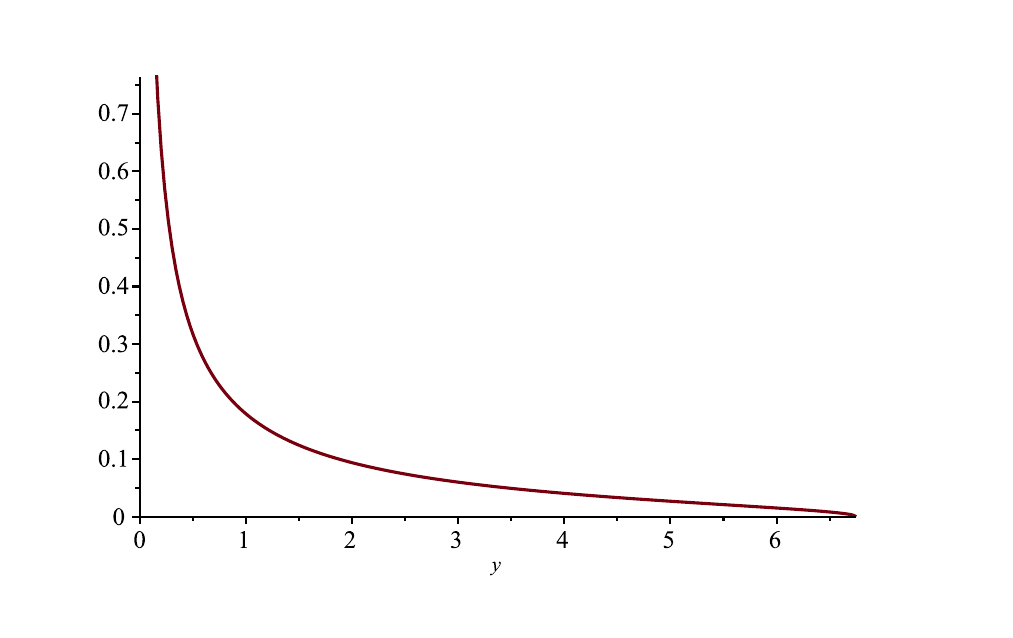}
\caption{The zero density for MOP with modified Bessel weights $(K_\nu,K_{\nu+1})$.}
\label{fig7}
\end{figure}

If $\alpha/n \to a >0$ then
\[    S(u) = \frac{1}{(1+a+u)^2}  \]
so that we get the free multiplicatve convolution $\mu_2 \boxtimes \mu_2$ where $\mu_2$ is the {Marchenko-Pastur} distribution MP$(a)$ on $[(\sqrt{a+1}-1)^2,(\sqrt{a+1}+1)^2]$.
If also $\nu/n \to b >0$ then
\[    S(u) = \frac{1}{(1+a+u)(1+a+b+u)}  \]
so that we get the free multiplicative convolution $\mu_2 \boxtimes \mu_3$ where $\mu_3$ is  the {Marchenko-Pastur} distribution MP$(a+b)$ on $[(\sqrt{a+b+1}-1)^2,(\sqrt{a+b+1}+1)^2]$. In both cases the zeros move away from the origin.

\subsection{Multiple OP with modified Bessel weights $(I_\nu,I_{\nu+1})$}
The modified Bessel function $I_\nu$ is unbounded and hence, in order to have finite moments, we need an exponential factor $e^{-cx}$, with $c>0$.
In \cite{Coussement-WVA1} multiple orthogonal polynomials were introduced for which
\begin{eqnarray*}
   \int_0^\infty  P_{n,m}(x) x^k e^{-cx} I_\nu(2\sqrt{x}) \,dx = 0,  &&  \qquad 0 \leq k \leq n-1, \\
  \int_0^\infty  P_{n,m}(x) x^k e^{-cx} I_{\nu+1}(2\sqrt{x}) \,dx = 0,  &&  \qquad 0 \leq k \leq m-1. 
\end{eqnarray*}
The polynomials with multi-indices near the diagonal have an explicit expression \cite{Coussement-WVA1}. If $p_{2n}(x) = P_{n,n}(x)$ and
$p_{2n+1}(x) = P_{n+1,n}(x)$, then
\[     p_n(x) = \sum_{k=0}^n (-1)^k \binom{n}{k} c^{-2k}  {}_2F_0\left( \begin{array}{c}   -k, \nu+n-k+1 \\  - \end{array}; -c  \right) x^{n-k}.  \]
The coefficients of these polynomials are related to Charlier polynomials $C_n(x;a)$ \cite[\S 9.14]{koekoek} since one has
\begin{equation}  \label{coefCharlier}
   {}_2F_0\left( \begin{array}{c}   -k, \nu+n-k+1 \\  - \end{array}; -c  \right) = C_k(-\nu-n+k-1;1/c) .
\end{equation}
For the asymptotic behavior of the coefficients we use the following result.

\begin{lemma}
The following limit holds
\[    \lim_{k/n \to t}   \frac{1}{n}  \frac{{}_2F_0\left( \begin{array}{c}   -k, \nu+n-k+1 \\  - \end{array}; -c  \right)}
                                                               { {}_2F_0\left( \begin{array}{c}   -k+1, \nu+n-k+2 \\  - \end{array}; -c  \right)}  = c (1-t).    \]
\end{lemma}

\begin{proof}
Let
\[   f(n,k) =  {}_2F_0\left( \begin{array}{c}   -k, \nu+n-k+1 \\  - \end{array}; -c  \right), \]
then one has the recurrence
\begin{equation}   \label{f(n,k)rec}
    f(n,k) + (2ck-cn-c\nu-2c-1) f(n,k-1) + c^2(k-1)(-\nu-n+k-2) f(n,k-2).  
\end{equation}
This can be obtained by using the recurrence relation for Charlier polynomials, combined with the difference relations for Charlier polynomials,
or simply by using the \texttt{hypersum} command in the maple package \texttt{sumtools}. Here is a proof using Charlier polynomials.
The three term recurrence relation for the Charlier polynomials is \cite[Eq (9.14.3) on p.~247]{koekoek}
\begin{equation}  \label{Charlier3TR}
  -x C_n(x;a) = a C_{n+1}(x;a) - (n+a) C_n(x;a) + n C_{n-1}(x;a).   
\end{equation}
Take $n=k$, $a=1/c$ and $x=-\nu-n+k-1$ to find
\begin{multline*}  c(\nu+n-k+1) C_k(-\nu-n+k-1;1/c) \\
= C_{k+1}(-\nu-n+k-1;1/c) - (kc+1) C_k(-\nu-n+k-1;1/c) + kc C_{k-1}(-\nu-n+k-1;1/c). 
\end{multline*}
For  $C_{k+1}$ we use the backward shift operator \cite[Eq. (9.14.8) on p.~248]{koekoek} 
\[      C_{k+1}(-\nu-n+k;1/c) - C_{k+1}(-\nu-n+k-1;1/c) = -(k+1) c C_k(-\nu-n+k-1;1/c),  \]
and for $C_{k-1}$ we use the forward shift operator \cite[Eq. (9.14.6) on p.~248]{koekoek} 
\[     C_{k-1}(-\nu-n+k-1;1/c) - c(-\nu-n+k-1) C_{k-1}(-\nu-n+k-2;1/c) = C_k(-\nu-n+k-1;1/c).  \]
Together with the relation \eqref{coefCharlier} we then find the recurrence \eqref{f(n,k)rec}.

Divide the recurrence \eqref{f(n,k)rec} by $n f(n,k-1)$, then the limit $L(t) = \lim_{k/n \to t} \frac{1}{n} \frac{f(n,k)}{f(n,k-1)}$ satisfies
\[  L(t) + c(2t-1) - c^2t(1-t) \frac{1}{L(t)} = 0.  \]
Solving this gives
\[   L(t) = \frac{-c(2t-1) \pm c}{2}.  \]
Note that $f(n,k)$ is a positive quantity since
\[     f(n,k) = \sum_{j=0}^k \binom{k}{j} (\nu+n-k+1)_j c^j, \]
so that we need to choose the positive solution, giving $L(t) = c(1-t)$.
\end{proof}

Using this result, we see that for the scaled polynomials $p_n(nx)$ one has
\[  a_{k,n} = n^{n-k} \binom{n}{k} c^{-2k} f(n,k), \]
so that
\[    f(t) = \lim_{k/n \to t} \frac{a_{k,n}}{a_{k-1,n}} = \frac{1-t}{t} c^{-2} c(1-t) = \frac{(1-t)^2}{c t}.  \]
The $S$-transform of the asymptotic zero distribution then is equal to
\[   S(u) = - \frac{1+u}{u} \frac{1}{f(-u)} = \frac{c}{1+u}.  \]
This is, up to the factor $c$, the $S$-transform of the Marchenko-Pastur distribution on $[0,4]$, which means that the
asymptotic zero distribution is the scaled Marchenko-Pastur distribution with density
\[     w(x) = \frac{c}{2\pi} \frac{\sqrt{4-cx}}{\sqrt{cx}}, \qquad x \in [0,\frac{4}{c}].  \]
This confirms the result in \cite[Corollary 1]{Coussement-WVA2}. 

These multiple orthogonal polynomials appear in the analysis of non-intersection squared Bessel paths \cite{kuijlaars-MF-Wiel} where they
give the average of the positions of the $n$ points at a fixed time. In this application one uses a different scaling and one
needs the asymptotic distribution of the zeros of $p(n^2 x)$ with the parameter $c =\hat{c}/n$ going to zero. Using the recurrence
\eqref{f(n,k)rec} once more but with $c = \hat{c}/n$, the limit $L(t) = \lim_{k/n \to t} \frac{f(n,k)}{f(n,k-1)}$ now satisfies
\[   L(t) + 2\hat{c}t-\hat{c} -1 - \hat{c}^2 t(1-t) 1/L(t) = 0, \]
so that
\[  L(t) = \frac{-2\hat{c}t+\hat{c}+1 \pm \sqrt{(\hat{c}+1)^2-4\hat{c}t}}{2}, \]
and one needs the positive solution, which corresponds to the $+$ sign. Therefore
\[   \lim_{k/n \to t}  \frac{f(n,k)}{f(n,k-1)} =  \frac{-2\hat{c}t+\hat{c}+1 + \sqrt{(\hat{c}+1)^2-4\hat{c}t}}{2}.  \]
Hence for the scaled polynomials $p_n(n^2x)$ we find 
\[    \lim_{k/n\to t} \frac{a_{k,n}}{a_{k-1,n}} = \frac{1-t}{t} \frac{1}{\hat{c}^2}  \frac{-2\hat{c}t+\hat{c}+1 + \sqrt{(\hat{c}+1)^2-4\hat{c}t}}{2} = f(t).  \]
The $S$-transform of the asymptotic distribution of the scaled zeros is then
\[    S(u) = \frac{2\hat{c}^2}{2\hat{c} u + \hat{c}+1 + \sqrt{(\hat{c}+1)^2+4\hat{c}u}} . \]
This can be written as
\[  S(u) = \frac{2\hat{c} u + \hat{c}+1 - \sqrt{(\hat{c}+1)^2+4\hat{c}u}}{2 u(u+1)}   = \frac{1}{u+1} \frac{2\hat{c} u + \hat{c}+1 - \sqrt{(\hat{c}+1)^2+4\hat{c}u}}{2 u}.  \]

\subsection{Multiple OP with hypergeometric weights}
Lima and Loureiro \cite{Lima-Loureiro22} introduced multiple orthogonal polynomials with respect to two  hypergeometric weights.
They used the weights $w_1(x)=w(x;a,b,c,d)$ and $w_2(x)= w(x;a,b+1,c+1)$, where
\[    w(x;a,b,c,d) = x^{a-1}(1-x)^{\delta-1} \ {}_2F_1(c-b;d-b;\delta;1-x), \quad  \delta=c+d-a-b.   \]
The multiple orthogonality is
\begin{eqnarray*}
  \int_0^1 p_n(x) w(x;a,b,c,d) x^k\, dx = 0, &&  \qquad 0 \leq k \leq \lceil \frac{n}2  \rceil \leq -1, \\
   \int_0^1 p_n(x) w(x;a,b+1,c+1,d) x^k\, dx = 0, &&  \qquad 0 \leq k \leq \lfloor \frac{n}2  \rfloor \leq -1,
\end{eqnarray*}
An explicit expression for the type II multiple orthogonal polynomials is
\[  p_n(x) = C_n \ {}_3F_2 \left( \left. \begin{array}{c} -n, c+ \lfloor \frac{n}{2} \rfloor, d+ \lfloor \frac{n-1}{2} \rfloor \\
                                                                                  a, b   \end{array}  \right| x \right) ,  \]
with $C_n$ a normalizing constant. Clearly
\[    a_{k,n} = \binom{n}{k} \frac{(c+\lfloor \frac{n}2 \rfloor)_{n-k} (d+\lfloor \frac{n}2 \rfloor)_{n-k}}{(a)_{n-k} (b)_{n-k}}, \]
so that
\[   \lim_{k/n \to t} \frac{a_{k,n}}{a_{k-1,n}} = \frac{1-t}{t} \frac{(1-t)^2}{(\frac32-t)^2} = f(t). \]
The $S$-transform of the asymptotic zero distribution then becomes
\[      S(u) = \frac{(u+\frac32)^2}{(u+1)^2},   \]
which is the $S$-transform of the free multiplicative convolution  $\mu_1 \boxtimes \mu_1$ with
\begin{equation}   \label{mu1}
       \mu_1 = \frac12 \delta_1 + \frac12\   \textup{KM}(0, 1).  
\end{equation}
This is the same asymptotic distribution as for the Jacobi-Pi\~neiro polynomials with $r=2$, as was already observed in \cite[\S 3.4]{Lima-Loureiro22}.

One can also allow the parameters to depend linearly on the degree $n$, leading to other $S$-transforms.

\subsection{Multiple OP with confluent hypergeometric weights}
Earlier, Lima and Loureiro \cite{Lima-Loureiro20} had studied multiple orthogonal polynomials with two confluent hypergeometric weights.
They used the weights $w_1(x) = w(x;a;b;c)$ and $w_2(x) = w(x;a,b,c+1)$, where $a>-1$, $b>-1$, $c> \max(0,a-b)$, and
\[  w(x;a,b,c) = µ\frac{\Gamma(c)}{\Gamma(a)\Gamma(b)} x^{a-1} e^{-x} U(c-b,a-b+1;x),          \]
where $U(\alpha,\beta;x)$ is the confluent hypergeometric function 
\[     U(\alpha,\beta;x) = \frac{1}{\Gamma(\alpha)} \int_0^\infty t^{\alpha-1} (1+t)^{\beta-\alpha-1} e^{-tx}\, dt.  \]
They found an explicit expression for the type II multiple orthogonal polynomials
\[     p_n(x) = \frac{(-1)^n (a+1)_n (b+1)_n}{(c+b+1+\lfloor \frac{n+d}{2} \rfloor )_n} \ {}_2F_2\left( \begin{array}{c} -n,c+b+1+\lfloor \frac{n+d}{2} \rfloor \\
                                                                                                                                                                                                       a+1, b+1 \end{array}; x \right), \]
so that
\[    a_{k,n} = \binom{n}{k} \frac{(a+1+n-k)_k (b+1+n-k)_k}{(c+b+1+\lfloor \frac{n+d}{2}\rfloor +n-k)_k}.  \]
Now we need to look at the scaled zeros $x_{k,n}/n$, and then for the polynomial $p_n(nx)$ we find
\[    \lim_{k/n \to t} \frac{a_{k,n}}{n a_{k-1,n}} = \frac{1-t}{t} \ \frac{(1-t)^2}{\frac32-t} = f(t).  \]
The $S$-transform of the asymptotic distrubtion of the scaled zeros is thus given by
\[    S(u) = \frac{u+\frac32}{(u+1)^2},  \]
which is a product of the $S$-transform $1/(u+1)$ of the Marchenko-Pastur distribution $\mu_0$ on $[0,4]$ and the measure $\mu_1$ given in \eqref{mu1} of the previous example, hence the asymptotic distribution of the scaled zeros is the free multiplicative convolution $\mu_0 \boxtimes \mu_1$.

Again one can allow the parameters to depend linearly on the degree $n$, but we leave this as an exercise.

\subsection{Multiple orthogonal polynomials with Meijer G-function weights}
Singular values of products of Ginibre random matrices form a determinantal point process in terms of multiple orthogonal polynomials
with Meijer G-function weights, \cite{kuijlaars-zhang}. The weights $(w_0,w_1,\ldots,w_{r-1})$ on $[0,\infty)$ are given by
\[  w_k(x) = \frac{1}{2\pi i} \int_{c-\infty}^{c+i\infty} (s+\nu_1)_k \prod_{j=1}^r \Gamma(s+\nu_j) \ x^{-s}\, ds,  \]
which is a Meijer G-function
\[   w_k(x) = G_{0,r}^{r,0} \left( \begin{array}{c}  - \\ \nu_r, \nu_{r-1}, \ldots , \nu_2 , \nu_1+k  \end{array} \Bigr| x \right), \]
with parameters $\nu_1,\nu_2,\ldots,\nu_r  > -1$. The type II multiple orthogonal polynomials on the stepline satisfy the orthogonality relations
\[  \int_0^\infty P_n(x) x^k w_j(x)\ dx = 0, \qquad k=0,1,\ldots,\lceil \frac{n-j}{r} \rceil -1, \]
for $j=0,\ldots,r-1$, so that $P_n$ is the type II multiple orthogonal polynomial $P_{\vec{n}}$ with multi-index 
\[  (\underbrace{m+1,\ldots,m+1}_{j \textup{ times}},m,\ldots,m)   \]
whenever $n=mr + j$. This polynomial has the hypergeometric representation \cite[p. 768]{kuijlaars-zhang}
\[    P_n(x) = (-1)^n \prod_{j=1}^r \frac{\Gamma (n+\nu_j+1)}{\Gamma(\nu_j+1)} \  {}_1F_r \left( \begin{array}{c} -n \\ \nu_1+1,\ldots,\nu_r+1 \end{array} \Bigr| x \right). \]
The ratio of the coefficients of $p_n(n^rx)$ is therefore
\[   \frac{a_{k,n}}{a_{k-1,n}} = \frac{(n-k+1)(\nu_1+n-k+1)\cdots(\nu_r+n-k+1)}{kn^r},  \]
so that
\[   f(t) = \lim_{k/n \to t}        \frac{a_{k,n}}{a_{k-1,n}} = \frac{(1-t)^{r+1}}{t},  \]
and the $S$-transform of the asymptotic distribution of the scaled zeros $\{x_{k,n}/n^r, 1 \leq k \leq n \}$ is
\[    S(u) = \frac{1}{(u+1)^r}.   \]
The asymptotic distrubution of the scaled zeros is therefore the $r$-fold multiplicative convolution $\mu_0 \boxtimes \mu_0 \boxtimes \cdots \boxtimes \mu_0$
 of the Marchenko-Pastur distribution $\mu_0$ on $[0,4]$. This is also known as the \textit{Fuss-Catalan} distribution.
For $r=2$ one finds Example \ref{ex:3.4} again.  

If $\nu_j/n \to a_j >0$ for $1 \leq j \leq  r$ then one easily finds
\[    S(u) = \frac{1}{(1+a_1+u)(1+a_2+u)\cdots(1+a_{r}+u)}, \]
which is the $S$-transform of the $r$-fold free multiplicative convolution $\mu_1\boxtimes \mu_2 \boxtimes \cdots \boxtimes \mu_{r}$,
where $\mu_j$ is the Marchenko-Pastur distribution MP$(a_j)$ on $[(\sqrt{a_j+1}-1)^2,(\sqrt{a_j+1}+1)^2]$.

\section{Conclusion}
We have shown that one can obtain the $S$-transform of the asymptotic distribution of (scaled) zeros of various families of polynomials by means of the asymptotic behavior of its coefficients and have illustrated this with examples involving orthogonal and multiple orthogonal polynomials. The $S$-transform informs us that this limit is
often the free multiplicative convolution of known measures, such as the Marchenko-Pastur distribution or the Kesten-McKay distribution.
Unfortunately there is no explicit expression for $\mu_1 \boxtimes \mu_2$ in terms of the measures $\mu_1$ and $\mu_2$. Therefore one has to work back to the
Stieltjes transform $G$ and then use the inversion formula for Stieltjes transforms to get an explicit expression for the density. If the $S$-transform is a rational function,
as was the case in most of the examples, then the Stieltjes transform will be an algebraic function and then the Stieltjes inversion will involve a study of the 
branch points and cuts of the corresponding Riemann surface. Sometimes the algebraic equation can be solved explicitly to find the Stieltjes transform,
see for example \cite{NeuschelWVA} for the algebraic equation coming from Jacobi-Pi\~neiro polynomials and multiple Laguerre polynomials of the first kind.

\end{document}